\documentclass[a4paper,14pt, reqno]{amsart}
\usepackage{bbm}

\numberwithin{equation}{section}

\usepackage{mathrsfs}
\usepackage{amsmath,amssymb, amsthm}
\usepackage[all,poly,knot]{xy}
\usepackage[colorlinks,linkcolor=black,anchorcolor=black,citecolor=black]{hyperref}
\usepackage{color}

\newtheorem{thm}{Theorem}[section]
\newtheorem{defi}[thm]{{Definition}}
\newtheorem{cor}[thm]{{Corollary}}
\newtheorem{lem}[thm]{{Lemma}}
\newtheorem{prop}[thm]{Proposition}

\newtheorem{note}[thm]{{Notation}}

\def\C{\mathscr C}
\def\Ac{\mathrm{Aut}_{\C}}

\def\Hom{\mathrm{Hom}}
\def\Homc{\Hom_{\C}}

\def\Id{\mathrm{Id}}

\def\Mor{\mathrm{Mor}}
\def\Obj{\mathrm{Obj}}

\def\lbr{\left(\begin{array}{c}}
\def\lbrt{\left(\begin{array}{cc}}
\def\lbrth{\left(\begin{array}{ccc}}
\def\rbr{\end{array}\right)}

\title[The MCM-approximation of the trivial module over a category algebra]{The MCM-approximation of the trivial module over a category algebra}
\author{Ren Wang}
\keywords{finite EI category, category algebra, Gorenstein-projective modules,  MCM-approximation} \subjclass[2010]{16G10, 16G50, 16D90, 18G25}
\address{School of Mathematical Sciences, University of Science and Technology of China, Hefei, Anhui 230026, P. R. China}
\email{renw@mail.ustc.edu.cn}
\date{\today}

\begin{document}
\begin{abstract}
 For a finite free EI category, we construct an explicit module over its category algebra. If in addition the category is projective over the ground field, the constructed module is Gorenstein-projective and is a maximal Cohen-Macaulay approximation of the trivial module. We give conditions on when the trivial module is Gorenstein-projective.
\end{abstract}

\maketitle

\section{Introduction}

Let $k$ be a field and $\C$ be a finite EI category. Here, the EI condition means
that all endomorphisms in $\C$ are isomorphisms. In particular,
$\Homc(x,x)=\Ac(x)$ is a finite group for each object $x$. Denote by $k\Ac(x)$ the group algebra.
Recall that
a finite EI category $\C$ is \emph{projective over $k$} if each
$k{\rm Aut}_{\C}(y)$-$k{\rm Aut}_{\C}(x)$-bimodule $k{\rm Hom}_{\C}(x,y)$ is projective on both sides; see~\cite[Definition 4.2]{WR}. 

The concept of a finite \emph{free} EI category is introduced in
\cite{LLi,LLi1}. Let $\C$ be a finite free EI category. For any morphism $x\overset{\alpha}{\rightarrow} y$ in $\C$, set 
	$V(\alpha)$ to be the set of objects $w$ such that there are factorizations $x\overset{\alpha'}{\rightarrow}
	w\overset{\alpha''}{\rightarrow}y$ of $\alpha$ with $\alpha''$ a non-isomorphism. 
	For any $w\in V(\alpha)$, we set $t_w(\alpha)=\alpha''\circ(\sum\limits_{g\in \Ac(w)} g)$, which is an element in $k\Homc(w,y)$. The freeness of $\C$ implies that the element $t_w(\alpha)$ is independent of the choice of $\alpha''$. 
	
	Denote by $k$-mod the category of finite dimensional $k$-vector spaces. We identify covariant functors from $\C$ to $k$-mod with left modules over the category algebra. We define a functor $E:\C\longrightarrow k\text{\rm -mod}$ as follows: for each object $x$, $E(x)$ is a $k$-vector space
	with basis $B_x=\bigsqcup\limits_{w\neq x} \Homc(w,x)\cup \{e_x\}$, and
	for each morphism $x\overset{\alpha}{\rightarrow} y$, the $k$-linear map $E(\alpha): E(x)\rightarrow E(y)$ sends $e_x$ to $e_y+\sum\limits_{w\in V(\alpha)} t_w(\alpha)$, and $\gamma$ to $\alpha\circ\gamma$
	for $\gamma\in \Homc(w,x)$ with $w\neq x$; see Proposition~\ref{EF}.

Recall that the \emph{constant functor} $\underline{k}: \C\rightarrow k$-mod is defined as follows: $\underline{k}(x)=k$ for each object $x$, and $\underline{k}(\alpha)=\Id_k$ for each morphism $\alpha$. The constant functor corresponds to the \emph{trivial module} of the category algebra. We mention that the trivial module $\underline{k}$ plays an important role in the cohomological study of categories~\cite{XF1}. It is the tensor identity in the category of modules over the category algebra~\cite{XF2}.

The notion of a \emph{Gorenstein-projective module} is introduced in~\cite{AB}; compare \cite{EJ}. They are generalizations of \emph{maximal Cohen-Macaulay modules} (MCM-modules for short) over Gorenstein rings. Hence, in the literature, Gorenstein-projective modules are often called MCM-modules.

Let $A$ be a finite dimensional algebra over $k$. For an $A$-module $X$, an \emph{{\rm MCM}-approximation} of $X$ is a map $\theta: G\rightarrow X$ with $G$ Gorenstein-projective such that any map $G'\rightarrow X$ with $G'$ Gorenstein-projective factors through $\theta$. The study of such approximations goes back to \cite{BuAs}. In general, the  MCM-approximation seems difficult to construct explicitly.
 
In this paper, we construct an explicit MCM-approximation of the trivial module, provided that the category $\C$ is free and projective over $k$. In this case, we observe that the category algebra is $1$-Gorenstein; see~\cite[Theorem 5.3]{WR}.
  We observe a surjective natural transformation $E\overset{\pi}{\longrightarrow} \underline{k}$ as follows: for each object $x$, $E(x) \overset{\pi_x}{\longrightarrow} \underline{k}(x)=k$ sends $e_x$ to $1_k$, and $\gamma$ to zero for $\gamma\in \Homc(w,x)$ with $w\neq x$.
\begin{thm}\label{MR}
Let $\C$ be a finite free EI category. Assume that the category $\C$ is projective over $k$. Then the map $E\overset{\pi}{\longrightarrow} \underline{k}$ is an MCM-approximation of the trivial module $\underline{k}$.
\end{thm}

This paper is organized as follows. In section 2, we recall some notation on finite free EI categories and construct the functor $E$. In section 3, we prove that $E$, viewed as a module over the category algebra, is Gorenstein-projective if in addition $\C$ is projective and prove Theorem~\ref{MR}. We mention that Proposition~\ref{FE} and Proposition~\ref{KERP} are new characterizations of finite free EI categories, and that Corollary~\ref{ST} describes when the trivial module is Gorenstein-projective.

\section{Finite free EI categories and the functor $E$}

Let $k$ be a field. Let $\C$ be a finite category, that is, it has
only finitely many morphisms, and consequently it has only finitely
many objects. Denote by $\Mor\C$ the finite set of all morphisms in
$\C$. Recall that the \emph{category algebra} \emph{k}$\C$ of $\C$ is defined as
follows: $\emph{k}\C=\bigoplus\limits_{\alpha \in \Mor\C}k\alpha$ as
a $k$-vector space and the product $*$ is given by the rule
\[\alpha * \beta=\left\{\begin{array}{ll}
\alpha\circ\beta, & \text{ if }\text{$\alpha$ and $\beta$ can be composed in $\C$}; \\
0, & \text{otherwise.}
\end{array}\right.\]
The unit is given by $1_{k\C}=\sum\limits_{x \in \Obj\C }\Id_x$,
where $\Id_x$ is the identity endomorphism of an object $x$ in $\C$.

If $\C$ and $\mathscr D$ are two equivalent finite categories, then
\emph{k}$\C$ and \emph{k}$\mathscr D$ are Morita equivalent; see
\cite[Proposition 2.2]{PWebb2}. In particular, $k\C$ is Morita
equivalent to $k\C_0$, where $\C_0$ is any skeleton of $\C$. 

In this paper, we assume that the finite category $\C$ is \emph{skeletal}, that is, any two distinct objects in $\C$ are not isomorphic.

\subsection{Finite free EI categories}
The category $\C$ is called a \emph{finite EI category} provided
that all endomorphisms in $\C$ are isomorphisms. In particular,
$\Homc(x,x)=\Ac(x)$ is a finite group for each object $x$ in $\C$. Denote by $k\Ac(x)$ the group algebra.

Let $\C$ be a finite EI category. Recall from ~\cite[Definition
2.3]{LLi} that a morphism $x\overset{\alpha}{\rightarrow} y$ in $\C$
is \emph{unfactorizable} if $\alpha$ is a non-isomorphism and
whenever it has a factorization $x
\overset{\beta}{\rightarrow} z \overset{\gamma}{\rightarrow} y$,
then either $\beta$ or $\gamma$ is an isomorphism. Let
$x\overset{\alpha}{\rightarrow} y$ be an unfactorizable
morphism. Then $h\circ\alpha\circ g$ is also unfactorizable for
every $h \in \Ac(y)$ and every $g \in \Ac(x)$; see \cite[Proposition
2.5]{LLi}. Any non-isomorphism $x\overset{\alpha}{\rightarrow} y$ in $\C$ has a decomposition
$x=x_0\overset{\alpha_1}{\rightarrow}
x_1\overset{\alpha_2}{\rightarrow} \cdots
\overset{\alpha_n}{\rightarrow} x_n=y$ with each $\alpha_i$
unfactorizable; see \cite[Proposition 2.6]{LLi}.

Following \cite[Definition 2.7]{LLi}, we say that a finite EI
category $\C$ satisfies the Unique Factorization Property (UFP), if
whenever a non-isomorphism $\alpha$ has two  decompositions into
unfactorizable morphisms:
\[ x=x_0\overset{\alpha_1}{\rightarrow}
x_1\overset{\alpha_2}{\rightarrow} \cdots
\overset{\alpha_m}{\rightarrow} x_m=y\] and
\[x=y_0\overset{\beta_1}{\rightarrow}
y_1\overset{\beta_2}{\rightarrow} \cdots
\overset{\beta_n}{\rightarrow} y_n=y,\] then $m=n$, $x_i=y_i$, and
there are $h_i\in \Ac(x_i)$, $1\leq i\leq n-1$ such that
$\beta_1=h_1\circ\alpha_1$, $\beta_2=h_2\circ\alpha_2\circ
h_1^{-1}$, $\cdots$, $\beta_{n-1}=h_{n-1}\circ\alpha_{n-1}\circ
h_{n-2}^{-1}$, $\beta_n=\alpha_n\circ h_{n-1}^{-1}$.

Let $\C$ be a finite EI category. Following \cite[Section 6]{LLi1},
we say that $\C$ is a finite \emph{free} EI category if it satisfies
the UFP. By ~\cite[Proposition 2.8]{LLi}, this is equivalent to the original definition ~\cite[Definition 2.2]{LLi}.

The following result is another characterization of a finite free EI category.

\begin{prop}\label{FE}
Let $\C$ be a finite EI category. Then $\C$ is free if and only if for any two  decompositions $x\overset{\alpha'}{\rightarrow} z\overset{\alpha''}{\rightarrow} y$ and  $x\overset{\beta'}{\rightarrow} w\overset{\beta''}{\rightarrow} y$ of a non-isomorphism $\alpha$ with $\alpha''$ and $\beta''$ non-isomorphisms, there is $z\overset{\gamma}{\rightarrow} w$ in $\C$ satisfying $\alpha''=\beta''\circ \gamma$ and $\beta'=\gamma \circ \alpha'$, or there is $w\overset{\delta}{\rightarrow} z$ in $\C$ satisfying $\beta''=\alpha''\circ \delta$ and $\alpha'=\delta \circ \beta'$.
\end{prop}

\begin{proof}
For the ``only if" part, we assume that $\C$ is a finite free EI category. Let $\alpha$ be a non-isomorphism with two  decompositions $x\overset{\alpha'}{\rightarrow} z\overset{\alpha''}{\rightarrow} y$ and  $x\overset{\beta'}{\rightarrow} w\overset{\beta''}{\rightarrow} y$ for $\alpha''$ and $\beta''$ non-isomorphisms. 

If $w=x$, we have that $\beta'$ is an isomorphism. Then we take $\delta=\alpha' \circ \beta'^{-1}$. If $z=x$, we have that $\alpha'$ is an isomorphism. Then we take $\gamma=\beta' \circ \alpha'^{-1}$. 

If $w,z\neq x$, we have that both $\alpha'$ and $\beta'$ are non-isomorphisms. Write the non-isomorphisms $\alpha',\alpha'',\beta',\beta''$ as compositions of unfactorizable
morphisms. Then we have two decompositions of $\alpha$ as follows
	\begin{equation*}
	\xymatrix@C=0.5cm{
		x\ar[r]^{\alpha'_1} \ar@{=}[d] & x_1 \ar[r]^{\alpha'_2}  &  \cdots \ar[r]^{\alpha'_m}  & x_m=z \ar[r]^{\alpha''_1} &  z_1 \ar[r]^{\alpha''_2} &  \cdots \ar[r]^{\alpha''_s} & z_s=y \ar@{=}[d] \\
		x\ar[r]^{\beta'_1} & w_1 \ar[r]^{\beta'_2} & \cdots \ar[r]^{\beta'_n}  & w_n=w \ar[r]^{\beta''_1} & y_1 \ar[r]^{\beta''_2} &  \cdots \ar[r]^{\beta''_t} & y_t=y,
	}
	\end{equation*}
where all $\alpha'_i, \alpha''_j, \beta'_l, \beta''_k$ are unfactorizable
morphisms and $\alpha'=\alpha'_m\circ \cdots \circ\alpha'_1$, $\alpha''=\alpha''_s\circ \cdots \circ\alpha''_1$, $\beta'=\beta'_n\circ \cdots \circ\beta'_1$, $\beta''=\beta''_t\circ \cdots \circ\beta''_1$. We apply the UFP for $\alpha$. If $m<n$, then $x_i=w_i$ for $1\leq i\leq m$, $z_j=w_{m+j}$ for $1\leq j\leq n-m$, $y_l=z_{n-m+l}$ for $1\leq l\leq t$, and there is $g: x_m\rightarrow x_m=w_m$ such that $\alpha''=\beta''\circ \beta'_n\circ \cdots \circ\beta'_{m+1}\circ g$ and $g\circ \alpha'=\beta'_m\circ \cdots \circ\beta'_1$. We take $\gamma=\beta'_n\circ \cdots \circ\beta'_{m+1}\circ g$. The cases $m=n$ and $m>n$ are similar.

For the ``if" part, let $\alpha$ be a non-isomorphism with the following two decompositions into
unfactorizable morphisms:
\[ x=x_0\overset{\alpha_1}{\rightarrow}
x_1\overset{\alpha_2}{\rightarrow} \cdots
\overset{\alpha_m}{\rightarrow} x_m=y\] and
\[x=y_0\overset{\beta_1}{\rightarrow}
y_1\overset{\beta_2}{\rightarrow} \cdots
\overset{\beta_n}{\rightarrow} y_n=y.\] Since $\alpha''=\alpha_m\circ \cdots \circ\alpha_2$ and $\beta''=\beta_n\circ \cdots \circ\beta_2$ are non-isomorphisms, by the hypothesis we may assume that there is $x_1 \overset{g_1}{\rightarrow} y_1$ in $\C$ such that $\beta_1=g_1\circ\alpha_1$ and $\alpha''=\beta''\circ g_1$. Since both $\alpha_1$ and $\beta_1$ are  unfactorizable, we infer that $g_1$ is an isomorphism and thus $x_1=y_1$. Then we have the following two decompositions  of $\alpha''$ into
unfactorizable morphisms:
\[ x_1\overset{\alpha_2}{\rightarrow}
x_2\overset{\alpha_3}{\rightarrow} \cdots
\overset{\alpha_m}{\rightarrow} x_m=y\] and
\[x_1\overset{\beta_2\circ g_1}{\rightarrow}
y_2\overset{\beta_3}{\rightarrow} \cdots
\overset{\beta_n}{\rightarrow} y_n=y.\]
Then we are done by repeating the above argument.
\end{proof}

\subsection{The functor $E$}
Let $\C$ be a finite free EI category. We will construct the functor $E:\C\longrightarrow k\text{\rm -mod}$ in the introduction, where $k$-mod is the category of finite dimensional $k$-vector spaces.

 Denote by $k\C$-mod the category of finite dimensional left modules over the category algebra $k\C$, and by ($k$-mod)$^{\C}$ the category of covariant functors from $\C$ to $k$-mod. There is a well-known equivalence $k\C$-mod $\simeq$ ($k$-mod)$^{\C}$; see~\cite[Theorem 7.1]{M}. We identify a $k\C$-module with a functor from $\C$ to $k$-mod. 

\begin{note}\label{V}
{\rm Let $\C$ be a finite free EI category. For any $x\overset{\alpha}{\rightarrow} y$ in $\C$, set 
\[V(\alpha)=\{w\in {\rm Obj}\C\mid \exists x\overset{\alpha'}{\rightarrow}
w\overset{\alpha''}{\rightarrow} y\   \text{such that}\ \alpha=\alpha''\circ \alpha' \ \text{with}\  \alpha''\  \text{a non-isomorphism} \}.\]
If $\alpha$ is an isomorphism, $V(\alpha)=\emptyset$ .
For any $w\in V(\alpha)$, we define \[t_w(\alpha)=\alpha''\circ(\sum\limits_{g\in \Ac(w)} g)\in k\Homc(w,y),\] where the non-isomorphism $\alpha'':w\rightarrow y$ is given by a factorization $x\overset{\alpha'}{\rightarrow}
w\overset{\alpha''}{\rightarrow}y$ of $\alpha$. By Proposition~\ref{FE}, such a morphism $\alpha''$ is unique up to an automorphism of $w$. It follows that the element $t_w(\alpha)$ is independent of the choice of $\alpha''$.}
\end{note}

\begin{lem}\label{VC}
Let $\C$ be a finite free EI category. For any $x\overset{\alpha}{\rightarrow} y$ and $y\overset{\beta}{\rightarrow} z$ in $\C$, we have
\[V(\beta\circ \alpha)=V(\beta)\sqcup V(\alpha),\] where the right hand side is a disjoint union.
\end{lem}

\begin{proof}
First, we prove that $V(\beta)\cap V(\alpha)=\emptyset$. Assume that there is a $w\in V(\beta)\cap V(\alpha)$. Then there are $x\overset{\alpha'}{\rightarrow}
w\overset{\alpha''}{\rightarrow}y$ and $y\overset{\beta'}{\rightarrow}
w\overset{\beta''}{\rightarrow}z$ such that $\alpha=\alpha''\circ \alpha'$ and $\beta=\beta''\circ \beta'$ with $\alpha'',\beta''$ non-isomorphisms. Since $\beta'\circ \alpha''\in \Ac(w)$, we have that $\alpha''$ is an isomorphism, which is a contradiction. Hence  $V(\beta)\cap V(\alpha)=\emptyset$.

For any $w\in V(\beta)$, there exist $y\overset{\beta'}{\rightarrow}
w\overset{\beta''}{\rightarrow}z$ such that $\beta=\beta''\circ \beta'$ with $\beta''$ a non-isomorphism. Then we have that $\beta\circ \alpha=\beta''\circ (\beta'\circ \alpha)$ with $\beta''$ a non-isomorphism, that is, $w\in V(\beta\circ \alpha)$. For any $w\in V(\alpha)$, there exist $x\overset{\alpha'}{\rightarrow}
w\overset{\alpha''}{\rightarrow}y$ such that $\alpha=\alpha''\circ \alpha'$ with $\alpha''$ a non-isomorphism. Then we have that $\beta\circ \alpha=(\beta\circ \alpha'')\circ \alpha'$. Since $\alpha''$ is a non-isomorphism, we have that $\beta\circ \alpha''$ is a non-isomorphism. Then we have that $w\in V(\beta\circ \alpha)$. In summary, we have $V(\beta)\sqcup V(\alpha)\subseteq V(\beta\circ \alpha)$.

 Let $w\in V(\beta\circ \alpha)$. Then there exist $x\overset{\eta'}{\rightarrow}
w\overset{\eta''}{\rightarrow}z$ in $\C$ such that $\beta\circ \alpha=\eta''\circ \eta'$ with $\eta''$ a non-isomorphism. Assume that $w=y$. By the non-isomorphism $\eta''$, we have that $w\neq z$. Then $y\neq z$ and $\beta$ is a non-isomorphism. Hence we have $y\in V(\beta)$ and thus $w\in V(\beta)$. 

Assume that $w\neq y$. If $\beta$ is an isomorphism, we have that $\alpha=(\beta^{-1}\circ \eta'')\circ \eta'$ with $\beta^{-1}\circ \eta''$ a non-isomorphism. This implies that $w\in V(\alpha)$. Assume that $\beta$ is a non-isomorphism. By Proposition~\ref{FE}, we have that there is a non-isomorphism $y\overset{\gamma}{\rightarrow} w$ satisfying $\beta=\eta''\circ \gamma$ or there is a non-isomorphism $w\overset{\delta}{\rightarrow} y$ satisfying $\alpha=\delta\circ \eta'$. This implies that $w\in V(\beta)$ or $w\in V(\alpha)$, since $\eta''$ and $\delta$ are non-isomorphisms. This proves $V(\beta\circ \alpha)\subseteq V(\beta)\sqcup V(\alpha)$, and we are done.
\end{proof}

\begin{lem}\label{TwE}
	Let $\C$ be a finite free EI category and $x\overset{\alpha}{\rightarrow} y$, $y\overset{\beta}{\rightarrow} z$ be two morphisms in $\C$. For  $w\in V(\beta)\subseteq V(\beta\circ \alpha)$, we have $t_w(\beta\circ \alpha)=t_w(\beta)$. For $w\in V(\alpha)\subseteq V(\beta\circ \alpha)$, we have $t_w(\beta\circ \alpha)=\beta\circ t_w(\alpha)$.
\end{lem}

\begin{proof}
For an object $w$ in $\C$, we set $G_w=\Ac(w)$. If $w\in V(\beta)$, then there exist $y\overset{\beta'}{\rightarrow}
w\overset{\beta''}{\rightarrow}z$ in $\C$ such that $\beta=\beta''\circ \beta'$ with $\beta''$ a non-isomorphism. By definition we have $t_w(\beta)=\beta''\circ(\sum\limits_{g\in G_w} g)$. By $\beta\circ \alpha=\beta''\circ (\beta'\circ \alpha)$, we have $t_w(\beta\circ \alpha)=\beta''\circ(\sum\limits_{g\in G_w} g)=t_w(\beta)$. 

If $w\in V(\alpha)$, then there exist $x\overset{\alpha'}{\rightarrow}
w\overset{\alpha''}{\rightarrow}y$ in $\C$ such that $\alpha=\alpha''\circ \alpha'$ with $\alpha''$ a non-isomorphism. By definition, we have $t_w(\alpha)=\alpha''\circ(\sum\limits_{g\in G_w} g)$. We observe that $\beta\circ \alpha=(\beta \circ \alpha'')\circ \alpha'$ with $\beta \circ \alpha''$ a non-isomorphism, since $\alpha''$ is a non-isomorphism. Then we have $t_w(\beta\circ \alpha)=(\beta\circ \alpha'')\circ(\sum\limits_{g\in G_w} g)=\beta\circ t_w(\alpha)$. 
\end{proof}
\begin{defi}\label{TDOFE}
	Let $\C$ be a finite free EI category. We define
	\[E:\C\longrightarrow k\text{\rm -mod}\] 
	as follows: for any $x\in \Obj\C$, $E(x)$ is a $k$-vector space
	with basis \[B_x=\bigsqcup\limits_{w\neq x} \Homc(w,x)\cup \{e_x\};\]
	for any $x\overset{\alpha}{\rightarrow} y$ in $\C$, the $k$-linear map $E(\alpha):E(x)\rightarrow E(y)$ is given by \[E(\alpha)(e_x)=e_y+\sum\limits_{w\in V(\alpha)} t_w(\alpha)\] and \[E(\alpha)(\gamma)=\alpha\circ\gamma\]
	for $\gamma\in \Homc(w,x)$ with $w\neq x$.
\end{defi}

\begin{prop}\label{EF}
		Let $\C$ be a finite free EI category. The above $E:\C\longrightarrow k\text{\rm -mod}$ is a functor.
\end{prop}
\begin{proof}
	For any $x\overset{\alpha}{\rightarrow} y\overset{\beta}{\rightarrow} z$ in $\C$, since $E(\beta\circ\alpha)(\gamma)=\beta\circ\alpha\circ \gamma=E(\beta)( E(\alpha)(\gamma))$ for $\gamma\in \Homc(w,x)$ with $w\neq x$. It suffices to prove that $E(\beta\circ\alpha)(e_x)=E(\beta)(E(\alpha)(e_x))$. We have \[E(\beta\circ\alpha)(e_x)=e_z+\sum\limits_{w\in V(\beta\circ\alpha)} t_w(\beta\circ\alpha)\] and \[E(\beta)(E(\alpha)(e_x))=e_z+\sum\limits_{w\in V(\beta)} t_w(\beta)+\beta\circ(\sum\limits_{w\in V(\alpha)} t_w(\alpha)).\] Then we are done by Lemma~\ref{VC} and Lemma~\ref{TwE}.
\end{proof}

We now introduce subfunctors of $E$ for later use. Let $\C$ be a finite free EI category. By the EI property, we may assume that $\Obj\C=\{x_1,x_2,\cdots,x_n\}$  satisfying
$\Homc(x_i,x_j)=\emptyset$ if $i<j$.
\begin{note}\label{SUBFE}
	For each $1\leq t\leq n$, we define a functor 
	\[Y^t:\C\longrightarrow k\text{\rm -mod}\] 
	as follows: $Y^t(x_i)$ is a $k$-vector space
	with basis \[B^t_i=\bigsqcup\limits_{l=i+1}^t \Homc(x_l,x_i)\cup \{e_{x_i}\}\] for $1\leq i\leq t$
	and $Y^t(x_i)=0$ for $i>t$;
	for any $x_j\overset{\alpha}{\rightarrow} x_i$ with  $j\leq t$, the $k$-linear map $Y^t(\alpha):Y^t(x_j)\rightarrow Y^t(x_i)$ is given by  \[Y^t(\alpha)(e_{x_j})=e_{x_i}+\sum\limits_{x_l\in V(\alpha)} t_{x_l}(\alpha)\] and \[Y^t(\alpha)(\gamma)=\alpha\circ\gamma\]
	for $\gamma\in B^t_j\backslash \{e_{x_j}\}$; $Y^t(\alpha)=0$ if $j>t$.
\end{note}

We observe that $Y^n=E$ and that $Y^t$ is a subfunctor of $Y^{t+1}$ for each $1\leq t\leq n-1$.

\subsection{An exact sequence}
We will describe a subfunctor of the functor $E$ such that the quotient functor is isomorphic to the constant functor.

Recall that the \emph{constant functor} $\underline{k}: \C\rightarrow k$-mod is defined by $\underline{k}(x)=k$ for all $x\in \Obj\C$ and $\underline{k}(\alpha)=\Id_k$ for all $\alpha\in \Mor\C$. The corresponding $k\C$-module is called the \emph{trivial module}.

\begin{note}\label{KER}
	Let $\C$ be a finite EI category. We define a functor 
	\[K:\C\longrightarrow k\text{\rm -mod}\] 
	as follows: for any $x\in \Obj\C$, $K(x)$ is a $k$-vector space
	with basis \[B'_x=\bigsqcup\limits_{w\neq x} \Homc(w,x);\]
	for any $x\overset{\alpha}{\rightarrow} y$ in $\C$, the $k$-linear map $K(\alpha):K(x)\rightarrow K(y)$ is given by \[K(\alpha)(\gamma)=\alpha\circ\gamma\]
	for $\gamma\in \Homc(w,x)$ with $w\neq x$.
\end{note}
Let $\C$ be a finite free EI category. We observe that $K$ is a subfunctor of $E$. 

 We have a surjective natural transformation $E\overset{\pi}{\longrightarrow} \underline{k}$ as follows: for any $x\in \Obj\C$, $E(x) \overset{\pi_x}{\longrightarrow} \underline{k}(x)=k$ sends $e_x$ to $1_k$, and $\gamma$ to zero for $\gamma\in \Homc(w,x)$ with $w\neq x$. Then we have an exact sequence of functors
 \begin{align}\label{SES}
0\longrightarrow K \overset{{\rm inc}}{\longrightarrow} E\overset{\pi}{\longrightarrow} \underline{k}\longrightarrow 0.
\end{align}

In what follows, we study when the above exact sequence splits, or equivalently, the epimorphism $\pi$ splits.

Recall that a finite category $\C$ is \emph{connected} if for any two distinct objects $x$ and $y$, there is a sequence of objects $x=x_0,x_1,\cdots,x_m=y$ such that either $\Homc(x_i,x_{i+1})$ or $\Homc(x_{i+1},x_i)$ is not empty, $0\leq i\leq m-1$. We say that the category $\C$ has a \emph{smallest object} $z$, if $\Homc(z,x)\neq \emptyset$ for each object $x$.

\begin{prop}\label{ST1}
	Let $\C$ be a finite connected free EI category. Then we have the following two statements.
	\begin{enumerate}
	\item If the category $\C$ has a smallest object $z$ such that $\Homc(z,x)$ has only one $\Ac(z)$-orbit for each object $x$, then the epimorphism $\pi$ in (\ref{SES}) splits.
	\item Assume that $\Ac(x)$ acts freely on $\Homc(x,y)$ for any objects $x$ and $y$. 
	If $\pi$ splits, then the category $\C$ has a smallest object $z$ such that $\Homc(z,x)$ has only one $\Ac(z)$-orbit for each object $x$.
	\end{enumerate}
\end{prop}
\begin{proof}
	For (1), let $z$ be the smallest object. We have that $\Homc(z,x)=\alpha_x\circ\Ac(z)$, where $\alpha_x\in \Homc(z,x)$ for each object $x$. 
We define a natural transformation $s:\underline{k}\rightarrow E$ as follows: for any $x\in \Obj\C$, $\underline{k}(x)=k \overset{s_x}{\longrightarrow} E(x)$ sends $1_k$ to $E(\alpha_x)(e_z)=e_x+\sum\limits_{w\in V(\alpha_x)} t_w(\alpha_x)$. We observe that $\pi \circ s={\Id}_{\underline{k}}$.
	
	 To prove that $s$ is a natural transformation, it suffices to prove that $s_y(1_k)=E(\alpha)(s_x(1_k))$ for any morphism $x \overset{\alpha}{\rightarrow} y$. We observe that $s_y(1_k)=E(\alpha_y)(e_z)$ and $E(\alpha)(s_x(1_k))=E(\alpha)(E(\alpha_x)(e_z))=E(\alpha\circ\alpha_x)(e_z)$. Since $\Homc(z,y)=\alpha_y\circ\Ac(z)$, we have that $\alpha\circ\alpha_x=\alpha_y\circ g$ for some $g\in \Ac(z)$. Then we have that $E(\alpha_y)(e_z)=E(\alpha\circ\alpha_x\circ g^{-1})(e_z)=E(\alpha\circ\alpha_x)(E(g^{-1})(e_z))=E(\alpha\circ\alpha_x)(e_z)$, since $E(g^{-1})(e_z)=e_z$ by the construction of $E$.
	 
	For (2), assume that $\pi$ splits. Then there is a natural transformation $s:\underline{k}\rightarrow E$ with $\pi \circ s={\Id}_{\underline{k}}$. It follows that $\underline{k}(x)=k \overset{s_x}{\longrightarrow} E(x)$ sends $1_k$ to $e_x+\sum\limits_{\{\gamma:w\rightarrow x\mid w\neq x\}} c_{\gamma}\gamma$, $c_{\gamma}\in k$ for each $x\in \Obj\C$. 
	
	We recall that $\Obj\C$ can be viewed as a poset, where $x\leq y$ if and only if $\Homc(x,y)\neq \emptyset$. We claim that $\C$ has a smallest object. Otherwise, there are two distinct minimal objects $z$ and $z'$ such that there is an object $x$ satisfying $\Homc(z,x)\neq \emptyset$ and $\Homc(z',x)\neq \emptyset$; see Lemma~\ref{PO}. Since $z$ and $z'$ are minimal, we have that $s_z(1_k)=e_z$ and $s_{z'}(1_k)=e_{z'}$. Let $\alpha\in \Homc(z,x)$ and $\beta\in \Homc(z',x)$. We have that $E(\alpha)(s_z(1_k))=s_x(1_k)=E(\beta)(s_{z'}(1_k))$, since $s$ is a natural transformation. We observe that $E(\alpha)(s_z(1_k))=E(\alpha)(e_z)=e_x+\sum\limits_{w\in V(\alpha)} t_w(\alpha)$, and $E(\beta)(s_{z'}(1_k))=E(\beta)(e_{z'})=e_x+\sum\limits_{w\in V(\beta)} t_w(\beta)$. Then we have $\sum\limits_{w\in V(\alpha)} t_w(\alpha)=\sum\limits_{w\in V(\beta)} t_w(\beta)$. Since $\Ac(w)$ acts freely on $\Homc(w,x)$, then $t_w(\alpha)\neq 0$ and $t_w(\beta)\neq 0$. Then we infer that $V(\alpha)=V(\beta)$ and $t_w(\alpha)=t_w(\beta)$ for each $w\in V(\alpha)=V(\beta)$. Then $z\in V(\alpha)=V(\beta)$, and thus $\Homc(z',z)\neq \emptyset$. This is a contradiction.
	
	Denote by $z$ the smallest element in $\C$. Let $\alpha$ and $\beta$ be two morphisms in $\Homc(z,x)$ for any object $x$. It suffices to prove that $\alpha$ and $\beta$ are in the same  $\Ac(z)$-orbit. Since $s$ is a natural transformation, we have that $E(\alpha)(s_z(1_k))=s_x(1_k)=E(\beta)(s_z(1_k))$. We observe that $E(\alpha)(s_z(1_k))=E(\alpha)(e_z)=e_x+\sum\limits_{w\in V(\alpha)} t_w(\alpha)$, and $E(\beta)(s_z(1_k))=E(\beta)(e_z)=e_x+\sum\limits_{w\in V(\beta)} t_w(\beta)$. Then we have that $t_z(\alpha)=t_z(\beta)$, that is, $\alpha\circ(\sum\limits_{g\in \Ac(z)} g)=\beta\circ(\sum\limits_{g\in \Ac(z)} g)$. 
	Since $\Ac(z)$ acts freely on $\Homc(z,x)$, we have $\alpha=\beta\circ h$ for some $h\in \Ac(z)$. Then we are done.
\end{proof}

The following lemma is well known. For completeness, we give a proof.
\begin{lem}\label{PO}
Let $(I,\leq)$ be a connected finite poset without a smallest element. Then there are two distinct minimal elements with a common upper bound.
\end{lem}
\begin{proof}
For each $z\in I$, set $S_z=\{a\in I\mid a\geq z\}$. Then $I=\bigcup\limits_{\{z\in I\mid z \  \text{is\ minimal}\}} S_z$.	Assume on the contrary that any two distinct minimal elements do not have a common upper bound. Then $I=\bigcup\limits_{\{z\in I\mid z \  \text{is\ minimal}\}} S_z$ is a disjoint union. Let $x\in S_z$ with $z$ minimal, and $y\in I$. We claim that if $x$ and $y$ are comparable, then $y\in S_z$. Indeed, if $x\leq y$, then $z\leq x\leq y$ and thus $y\in S_z$. If $y\leq x$, we assume that $y\in S_{z'}$ with $z'$ minimal. Then
 $z'\leq y\leq x$, and thus $x\in S_{z'}$. Hence $z'=z$.

The above claim implies that each $S_z$ is a connected component of $I$. This contradicts to the connectedness of $I$.
\end{proof}

\section{The Gorenstein-projective module $E$}
 In this section, we prove that $E$, viewed as a module over the category algebra, is Gorenstein-projective if in addition $\C$ is projective and prove Theorem~\ref{MR}. For the proof, we study the subfunctor $K$ of $E$. We observe that the projectivity of $K$ is equivalent to the freeness of the category $\C$; see Proposition~\ref{KERP}.
 
\subsection{Gorenstein-projective modules}
Let $A$ be a finite dimensional algebra over $k$. Denote by $A$-mod the category of finite dimensional left $A$-modules. The opposite algebra of $A$ is denoted by $A^{\rm op}$. We identify right $A$-modules with left $A^{\rm op}$-modules.

 Denote by $(-)^*$ the contravariant functor ${\rm Hom}_A(-,A)$ or ${\rm Hom}_{A^{\rm op}}(-,A)$. Let $X$ be an $A$-module. Then $X^*$ is a right $A$-module and $X^{**}$ is a $A$-module. There is an evaluation map ${\rm ev}_X: X\rightarrow X^{**}$ given by ${\rm ev}_X(x)(f)=f(x)$ for $x\in X$ and $f\in X^*$. Recall that an $A$-module $G$ is \emph{Gorenstein-projective} provided that ${\rm Ext}^i_A(G,A)=0={\rm Ext}^i_{A^{\rm op}}(G^*,A)$ for $i\geq 1$ and the evaluation map ${\rm ev}_G$ is bijective; see~\cite[Proposition 3.8]{AB}. 
 
 Denote by $A$-Gproj the full subcategory of $A$-mod consisting of Gorenstein-projective $A$-modules.
 Recall that $A$-Gproj is closed under extensions, that is, for a short exact sequence $0\rightarrow X\rightarrow Y\rightarrow Z\rightarrow 0$ of $A$-modules, $X,Z \in A$-Gproj implies $Y \in  A$-Gproj. It is well known that ${\rm Ext}^1_A(G,X)=0$ for any $A$-module $X$ with finite projective dimension and $G\in A$-Gproj; see~\cite[Section 10.2]{EJ}.

 In the literature, Gorenstein-projective modules are also called \emph{maximal Cohen-Macaulay} (MCM for short) modules. For an $A$-module $X$, an \emph{{\rm MCM}-approximation} of $X$ is a map $\theta: G\rightarrow X$ with $G$ Gorenstein-projective such that any map $G'\rightarrow X$ with $G'$ Gorenstein-projective factors through $\theta$. We observe that such an MCM-approximation is necessarily surjective.
 
 By a \emph{special {\rm MCM}-approximation}, we mean an epimorphism $\theta: G\rightarrow X$ with the kernel of $\theta$, denoted by ${\rm ker}\theta$, has finite projective dimension; compare~\cite[Definition 7.1.6]{EJ}. Since ${\rm Ext}^1_A(G',{\rm ker}\theta)=0$ for any Gorenstein-projective $G'$, we have that a special MCM-approximation is an MCM-approximation.

The algebra $A$ is \emph{Gorenstein} if ${\rm id}_A A<\infty$ and ${\rm id} A_A<\infty$, where ${\rm id}$ denotes the injective dimension of a module. It is well known that for a Gorenstein algebra $A$ we have ${\rm id}_A A={\rm id} A_A$; see~\cite[Lemma A]{Zaks}. Let $m\geq 0$. A Gorenstein algebra $A$ is \emph{$m$-Gorenstein} if ${\rm id}_A A={\rm id} A_A\leq m$. 

The following result is well known; see~\cite[Lemma 4.2]{XWC}.
\begin{lem}\label{1SGP}
Let $A$ be a $1$-Gorenstein algebra. Then an $A$-module $G$ is Gorenstein-projective if and only if there is a monomorphism $G\rightarrow P$ with $P$ projective.
\end{lem}

Let $\C$ be a finite EI category which is skeletal. Recall from ~\cite[Definition 4.2]{WR} that
the category $\C$ is \emph{projective over $k$} if each
$k{\rm Aut}_{\C}(y)$-$k{\rm Aut}_{\C}(x)$-bimodule $k{\rm Hom}_{\C}(x,y)$ is projective on both sides. By ~\cite[Proposition 5.1]{WR} the category algebra $k\C$ is Gorenstein if and only if  $\C$ is projective over $k$, in which case, $k\C$ is $1$-Gorenstein if and only if the category $\C$ is free; see~\cite[Theorem 5.3]{WR}.

\subsection{Functors as modules}
Let us recall from ~\cite[III.2]{ARS} some notation on matrix algebras.
 Let $n\geq 2$. Recall that an $n\times n$ upper
 triangular matrix algebra $\Gamma=\left(
 \begin{array}{cccc}
 A_1 & M_{12} & \cdots & M_{1n} \\
 & A_2 & \cdots & M_{2n} \\
 &  & \ddots & \vdots \\
 &  &  & A_n \\
 \end{array}
 \right)$ is given by the following data: each $A_i$ is a finite dimensional algebra over $k$ , each $M_{ij}$ is an $A_i$-$A_j$-bimodule on which $k$ acts centrally for $1\leq i< j\leq n$. The multiplication of the matrix algebra is induced by $A_i$-$A_j$-bimodule morphisms $\psi_{ilj}: M_{il}\otimes_{R_l} M_{lj} \rightarrow M_{ij}$, which satisfy the following identities
  \[\psi_{ijt}(\psi_{ilj}(m_{il}\otimes m_{lj})\otimes m_{jt})=\psi_{ilt}(m_{il}\otimes \psi_{ljt}(m_{lj}\otimes m_{jt})),\]  for $1\leq i<l<j<t\leq n$.

  Recall that a left $\Gamma$-module $X=\lbr X_1 \\
 \vdots \\ X_n \\ \rbr$ is described by a column vector: each $X_i$ is a left $A_i$-module, the left $\Gamma$-module structure is induced by left $A_j$-module morphisms $\varphi_{jl}:
 M_{jl}\otimes_{R_l} X_l \rightarrow X_j$, which satisfy the following identities
 \[\varphi_{ij}\circ({\rm Id}_{M_{ij}}\otimes \varphi_{jl})=\varphi_{il}\circ(\psi_{ijl}\otimes {\rm Id}_{X_l}),\] for $1\leq i<j<l\leq n$.

Let $\C$ be a skeletal finite EI category with $\Obj\C=\{x_1,x_2,\cdots,x_n\}$ satisfying
 $\Homc(x_i,x_j)=\emptyset$ if $i<j$.
 Set $M_{ij}=k\Homc(x_j,x_i)$.
 Write $A_i=M_{ii}$, which is the group algebra of $\Ac(x_i)$. Recall that the category algebra $k\C$ is isomorphic to the corresponding
 upper triangular matrix algebra $\Gamma_{\C}=\left(
 \begin{array}{cccc}
 A_1 & M_{12} & \cdots & M_{1n} \\
 & A_2 & \cdots & M_{2n} \\
 &  & \ddots & \vdots \\
 &  &  & A_n \\
 \end{array}
 \right)$; see~\cite[Section 4]{WR}. The corresponding maps $\psi_{ilj}$ are induced by the composition of morphisms in $\C$.
 
  Then we have the following equivalences
 \[(k\text{-mod})^{\C}\overset{\sim}{\longrightarrow} k\C\text{-mod}\overset{\sim}{\longrightarrow}\Gamma_{\C}\text{-mod},\]
 which sends a functor $X:\C\rightarrow k$-mod to a $\Gamma_{\C}$-module $\lbr X_1\\ \vdots \\
 X_n\rbr$ as follows: $X_i=X(x_i)$, and $\varphi_{ij}:M_{ij}\otimes_{A_j} X_j\rightarrow X_i$ sends $\alpha\otimes a_j$ to $X(\alpha)(a_j)$ for $\alpha\in \Homc(x_j,x_i)$ and $a_j\in X_j$ for $1\leq i<j\leq n$. In what follows, we identify these three categories.
 
\begin{prop}\label{EGP}
	Let $\C$ be a finite free EI category and $E:\C\rightarrow k${\rm -mod} be the functor in Definition~\ref{TDOFE}. Assume that $\C$ is projective. Then $E$, viewed as a $k\C$-module, is Gorenstein-projective.
\end{prop}
\begin{proof}
For each $1\leq t\leq n$, let $Y^t$ be the functor in Notation~\ref{SUBFE}. Then we have a filtration $0=Y^0\subseteq Y^1\subseteq \cdots \subseteq Y^{n-1}\subseteq Y^n=E$ of subfunctors. Since Gorenstein-projective modules are closed under extensions, it suffices to prove that each $Y^t/Y^{t-1}$ is a Gorenstein-projective $\Gamma_{\C}$-module for $1\leq t\leq n$. 

Recall from ~\cite[Theorem 5.3]{WR} that $\Gamma_{\C}$ is $1$-Gorenstein. 
We observe that $Y^1\simeq \lbr ke_{x_1}\\ 0 \\ \vdots \\0\rbr$, and $Y^t/Y^{t-1}\simeq \lbr M_{1t} \\ \vdots \\ M_{t-1,t} \\ke_{x_t}\\0\\ \vdots \\ 0\rbr$ as $\Gamma_{\C}$-modules, where the structure maps $\varphi_{ij}$ of $Y^t/Y^{t-1}$ are described as follows: $\varphi_{ij}=\psi_{ijt}$ if $i<j<t$, $\varphi_{it}(\alpha\otimes e_{x_t})=\alpha\circ(\sum\limits_{g\in \Ac(x_t)} g)$ for $\alpha\in \Homc(x_t,x_i)$, and $\varphi_{ij}=0$ if $j>t$.

Denote by $C_t$ the $t$-th column of $\Gamma_{\C}$. It is a projective $\Gamma_{\C}$-module. We observe that each $Y^t/Y^{t-1}$ is embedded in $C_t$, by sending $e_{x_t}$ to $\sum\limits_{g\in \Ac(x_t)} g$. By Lemma~\ref{1SGP}, each $Y^t/Y^{t-1}$ is Gorenstein-projective.
\end{proof}
 Let $\C$ be a skeletal finite EI category with $\Obj\C=\{x_1,x_2,\cdots,x_n\}$ satisfying
$\Homc(x_i,x_j)=\emptyset$ if $i<j$ and let $\Gamma_{\C}$ be the corresponding
upper triangular matrix algebra. 
For each $1\leq t\leq n$, denote by the $\Gamma_{\C}$-module
$i_t(R_t)^*=
\lbr M_{1t} \\
\vdots \\
M_{t-1,t} \\
0\\
\vdots\\
0 \rbr $, 
whose structure map is given by $\varphi_{ij}=\psi_{ijt}$ 
if $i<j<t$, and $\varphi_{ij}=0$ otherwise. Denote by $\Gamma_t$ the algebra given by the $t\times t$ leading
principal
submatrix of $\Gamma_{\C}$. Denote by $M_t^*$ the left
$\Gamma_t$-module $\lbr M_{1,t+1}\\
\vdots \\ M_{t,t+1}\\  \rbr$ for $1\leq t\leq n-1$. We use the same notation in ~\cite{WR}.

The following result is implicitly contained in~\cite[Proposition 4.5]{WR}.
\begin{prop}\label{KERP}
	Let $\C$ be a finite projective EI category and $K:\C\rightarrow k${\rm -mod} be the functor in Notation~\ref{KER}. Then $K$, viewed as a $k\C$-module, is projective if and only if the category $\C$ is free.
\end{prop}

\begin{proof}
	We observe that the functor $K:\C\rightarrow k${\rm -mod} in Notation~\ref{KER} is isomorphic to $\bigoplus\limits_{t=2}^n i_t(R_t)^*$ as $\Gamma_{\C}$-modules.
	
We have the fact that $i_t(R_t)^*$ is a projective $\Gamma_{\C}$-module if and only if $M^*_{t-1}$ is a projective $\Gamma_{t-1}$-module for each $1\leq t\leq n$; see~\cite[Lemma 2.1]{WR}. By~\cite[Proposition 4.5]{WR}, the category $\C$ is free if and only if each $M^*_t$ is a projective $\Gamma_t$-module for $1\leq t\leq n-1$. Then we are done by the above observation.
\end{proof}

\subsection{The proof of Theorem 1.1}
We now are in a position to prove Theorem~\ref{MR}. Recall the map $E\overset{\pi}{\longrightarrow} \underline{k}$ in~(\ref{SES}). The constant functor $\underline{k}$ corresponds to the trivial module of $k\C$.

\begin{thm}\label{NM}
Let $\C$ be a finite free EI category. Assume that the category $\C$ is projective over $k$. Then the map $E\overset{\pi}{\longrightarrow} \underline{k}$ is an MCM-approximation of the trivial module $\underline{k}$.
\end{thm}
\begin{proof}
 By Proposition~\ref{EGP}, the $k\C$-module $E$ is Gorenstein-projective. By Proposition~\ref{KERP}, the $k\C$-module $K$ is projective. Hence the exact sequence ~(\ref{SES}) is a special MCM-approximation of $\underline{k}$.
\end{proof}

\begin{cor}\label{ST}
	Let $\C$ be a finite connected EI category, which is free and projective over $k$. Then we have the following statements.
		\begin{enumerate}
			\item If the category $\C$ has a smallest object $z$ such that $\Homc(z,x)$ has only one $\Ac(z)$-orbit for each object $x$, then the trivial module $\underline{k}$ is Gorenstein-projective.
			\item Assume that $\Ac(x)$ acts freely on $\Homc(x,y)$ for any objects $x$ and $y$. 
			If $\underline{k}$ is Gorenstein-projective, then the category $\C$ has a smallest object $z$ such that $\Homc(z,x)$ has only one $\Ac(z)$-orbit for each object $x$.
		\end{enumerate}
\end{cor}
\begin{proof}
We observe that $\underline{k}$ is Gorenstein-projective if and only if the short exact sequence~(\ref{SES}) splits. Here, we use the fact that Gorenstein-projective modules are closed under direct summands. Consequently, the statement follows immediately from Proposition~\ref{ST1}.
\end{proof}

\section*{Acknowledgments}
The author is grateful to her supervisor Professor Xiao-Wu Chen for
his guidance. This work is supported by the National Natural
Science Foundation of China (No.s 11522113 and 11571329) and the Fundamental Research Funds for the Central Universities.

\end{document}